\documentclass[12pt]{article}

\usepackage{amssymb,amsmath,amsfonts,amssymb}
\usepackage{graphics,graphicx,color,hhline}
\usepackage{bbm} 
\usepackage{euscript}  
\usepackage{authblk} 

\def\@abssec#1{\vspace{.05in}\footnotesize \parindent .2in 
{\bf #1. }\ignorespaces} 

\graphicspath{{/EPSF/}{Figures/}}   

\newtheorem{theorem}{Theorem}[section]
\newtheorem{lemma}[theorem]{Lemma}
\newtheorem{proposition}[theorem]{Proposition}

\def \Rm {\mathbb R}
\def \Nm {\mathbb N}
\def \Cm {\mathbb C}
\def \Sm {\mathbb S}
\newcommand{\eps}{\varepsilon}
\newcommand{\E}{\mathbb E}
\newcommand{\ds}{\displaystyle}

\newcommand{\calQ}{\mathcal Q}
\newcommand{\calC}{\mathcal C}

\newcommand{\calL}{\mathcal L}

\newcommand{\calF}{\mathcal F}

\newcommand{\calO}{\mathcal O}

\newcommand{\calR}{\mathcal R}

\newcommand{\calD}{\mathcal D}

\def\fref#1{{\rm (\ref{#1})}}

\newcommand{\cout}[1]{}

\newcommand{\be}{\begin{equation}}
\newcommand{\ee}{\end{equation}}
\newcommand{\bea}{\begin{eqnarray}}
\newcommand{\eea}{\end{eqnarray}}
\newcommand{\bee}{\begin{eqnarray*}}
\newcommand{\eee}{\end{eqnarray*}}
\newcommand{\bal}{\begin{align*}}
\newcommand{\eal}{\end{align*}}

\begin{document}
{\title{Radiative transfer with long-range interactions: regularity and asymptotics}}

 \author[,1]{Christophe Gomez, \footnote{ christophe.gomez@univ-amu.fr}}
\affil[1]{Aix Marseille Universit{\'e}, CNRS, Centrale Marseille, I2M, UMR 7373, 13453 Marseille, France}
 \author[,2]{Olivier Pinaud \footnote{pinaud@math.colostate.edu}}
 \affil[2]{Department of Mathematics, Colorado State University, Fort Collins CO, 80523}
 \author[,3]{Lenya Ryzhik \footnote{ryzhik@math.stanford.edu}}
 \affil[3]{Department of Mathematics, Stanford University, Stanford CA 94305 }

\maketitle

\begin{abstract}
 This work is devoted to radiative transfer equations with long-range interactions. Such equations arise in the modeling of high frequency wave propagation in random media with long-range dependence. In the regime we consider, the singular collision operator modeling the interaction between the wave and the medium is conservative, and as a consequence wavenumbers take values on the unit sphere. Our goals are to investigate the regularizing effects of grazing collisions, the diffusion limit, and the peaked forward limit. As in the case where wavenumbers take values in $\Rm^{d+1}$, we show that the transport operator is hypoelliptic, so that the solutions are infinitely differerentiable in all variables. Using probabilistic techniques, we show as well that the diffusion limit can be carried on as in the case of a regular collision operator, and as a consequence that the diffusion coefficient is non-zero and finite. We finally consider the regime where grazing collisions are dominant.
\end{abstract}

\section{Introduction}
This work is the sequel to our first paper \cite{hypo} on radiative transfer equations with long-range interactions. More precisely, we are interested in transport equations of the form
\be \label{transportL}
\left\{
\begin{array}{l}
\ds \partial_t f+  k \cdot \nabla_x f=\calL f, \qquad \textrm{for } \qquad (t,x, k) \in \Rm_+ \times \Rm^{d+1} \times \Sm^d\\
\ds f(t=0,\cdot,\cdot)=f^0 \in L^2(\Rm^{d+1} \times \Sm^d)\\
\ds \calL f( k)= \int_{\Sm^d} F( k \cdot  p) \left(f( p)-f( k)\right)d\sigma(p).
\end{array} \right.
\ee
Above, $d \geq 1$, $\Sm^d$ is the unit sphere in $\Rm^{d+1}$ and $d\sigma(p)$ is the surface measure on $\Sm^d$. Here, the collision kernel $F$ is non-negative and non-integrable at zero so that
$$\int_{\Sm^d} F( k \cdot  p) d\sigma( p)=\infty.$$
Our motivation for studying such equations stems from the analysis of wave propagation in random media with long-range dependence. For high frequency waves in the weak coupling regime, the wave energy is asymptotically described by a transport equation of the form \fref{transportL} with a singular kernel. The fact that the collision process is confined to the sphere is due to the conservative nature of the interaction between the wave and the random medium. Other types of collision kernel arise depending on the scalings. The one we consider has the form 
$$
\calL f(k)=\int_{\Rm^{d+1}} \delta \left( \frac{|k|^2}{2}- \frac{|p|^2}{2}
\right) 
\hat R(k-p) (f(p)-f(k)) dp,
$$
where $\delta$ is the Dirac measure, that decouples the transport equations for different values of $|k|$, and we may set $|k|=1$. The term $\hat R$ is the power spectrum of a random potential $V$, mean-zero and statistically homogeneous in space,
\[
R(x)=\E [V(x+y)V(y)]=
\frac{1}{(2 \pi)^d}\int_{\Rm^{d+1}} e^{i k \cdot x} \hat R(k) dk, 
\]
 that models random fluctuations in the Schr\"odinger equation $$
i \partial_t \psi^\eps+ \frac{1 }{2}\Delta_x \psi^\eps - 
\sqrt{\eps} \,V\left({x} \right) \psi^\eps=0. 
$$
Here, ${\eps\ll 1}$ is the variance of the fluctuations. Equation \fref{transportL} then describes energy transport in a certain macroscopic limit via asymptotics of the Wigner transform \cite{GMMP,LP,wigner}. See \cite{balreview, Erdos-Yau,LUK, Spohn, gomez1} for more details on the link between \fref{transportL} and wave propagation in random media. 

In the present work, we investigate the regularity of the solutions of \fref{transportL} and two asymptotic regimes. We show that the transport operator is hypoelliptic, namely that for any initial condition in $L^2(\Rm^{d+1} \times \Sm^d)$, the solution is $\calC^\infty$ in all variables for positive times. This is a consequence of the singular nature of the collision operator when the correlation function $R(x)$ decays only
algebraically: 
\[
\textrm{when } \quad R(x)\sim\frac{1}{|x|^{2-2\alpha}},~~|x|\gg 1, \qquad \textrm{ then } \quad
\hat R(p)\sim\frac{1}{|p|^{2\alpha+d-1}},
\]
and this singularity is non-integrable when $\alpha\in(1/2,1)$ (recall
that the integration is carried over the
$d$-dimensional sphere). Therefore, we have (as $|k|=|p|=1$)
\[
\hat R (k-p)=\hat R (|k-p|)=\hat R (2
 \sqrt{1-\cos \theta}) \sim \theta^{-d-2\beta},\hbox{ as $\theta \to 0$,}
\]
where $\beta=2\alpha-1 \in (0,1)$. The regularizing effect of grazing collisions is now well-established in the case momenta take values in $\Rm^{d+1}$ and not just on the sphere, see e.g. \cite{AlexVillani,strain,strain2,alexandre_global,alexandre_hypo,alexandre_uncertainty,bouchut,lerner} for references on the Boltzmann equation in the so-called non cut-off case. The heuristics of the regularization effect is similar to the case of collisions over $\Rm^{d+1}$, and goes as follows: the high frequency behavior of the collision operator $\calL$ is essentially that of a fractional Laplace-Beltrami operator on the sphere, and as a consequence standard energy estimates yield some Sobolev regularity in the $k$ variable. This regularity is then propagated to the other variables by the free transport operator. The proof of regularity is based on the following hypoelliptic estimates that we derived in \cite{hypo}:
\begin{theorem} \label{hypo} Assume $g,h \in L^2(\Rm_t \times \Rm^{d+1}_x
  \times \Sm^d)$, and let $f \in L^2(\Rm_t \times \Rm^{d+1}_x \times
  \Sm^d)$ satisfy the transport equation 
\be \label{transporthypo}
\partial_t f+ k \cdot \nabla_x f= (-\Delta_d )^\beta h+g
\ee 
in the distribution sense, where $\Delta_d$ is the Laplace-Beltrami operator on $\Sm^d$ and $\beta \geq 0$. For some
  $\theta>0$, suppose, in addition, that 
\[
(-\Delta_d
  )^\frac{\theta}{2} f \in L^2(\Rm_t \times \Rm^{d+1}_x \times
  \Sm^d).
\] 
Then, for 
\[
\gamma=\frac{\theta}{2(1+2\beta)+\theta},
\]
we have
  $ \partial^\gamma_{t,x} f \in L^2(\Rm_t \times \Rm^{d+1}_x \times
  \Sm^d)$ with the estimate
$$
\| \partial_{t,x}^\gamma f \|_{L^2} \leq C \left( \|(-\Delta_d )^\frac{\theta}{2} f \|_{L^2}+\|f \|_{L^2}+\|g \|_{L^2}+\|h \|_{L^2} \right).
$$
\end{theorem}

The fractional derivatives above are defined in the Fourier space, see the notation section further. Note that we stated here a slightly more general version of the theorem of \cite{hypo} that better suits our needs for this work: we added an additional source term $g$ in \fref{transporthypo} which essentially does not modify the original proof. The proof of Theorem \ref{hypo} follows the techniques of Bouchut \cite{bouchut}, that are adapted to the spherical geometry.  Replacing formally $(-\Delta_d )^\beta h $ in the theorem by $\calL f$, it is expected, using bootstrapping arguments, that $f$ has derivatives of any order in all variables. 

We will also address two asymptotic regimes. The first one is the diffusion regime where the long time behavior of solutions to \fref{transportL} is investigated. The main question is whether the singular nature of the collision operator leads to a different equation than the standard diffusion equation obtained for smooth collision kernels. The answer is no with the singularities that we consider, we obtain a perfectly defined diffusion matrix. The second regime is the peaked forward regime where grazing collisions are the most significant. In the case where collision are defined over $\Rm^{d+1}$, this leads to the well-known Fokker-Planck (or fractional Fokker-Planck) equation. Here, the situation is slightly different, the limiting collision operator is not a fractional Laplace-Beltrami operator, but an operator that shares the same high frequency behavior as the Laplace-Beltrami, but not the same low frequency behavior. We will base our asymptotic analysis on the probabilistic representation of the transport equation \fref{transportL}. This is mostly motivated by (future) numerical considerations: while standard discretizations of \fref{transportL}, using e.g. finite elements of finite volumes, might not be straightforward due to the singularity of the kernel, probabilistic methods offer a simple alternative. The operator $\calL$ can be seen as the generator of a jump Markov process, which can be easily simulated. The solution to \fref{transportL} is then obtained after averaging over several realizations of the process.

The paper is structured as follows: in Section \ref{prelim}, we introduce the notation and state the hypotheses on the kernel $F$. We present our main results in Section \ref{res}. Theorem \ref{th1} concerns the regularity theory, Theorem \ref{th2} the diffusion limit, and Theorem \ref{th3} the peaked forward limit. The probabilistic representation of the solution to \fref{transportL} is introduced before Theorem \ref{th2}. Section \ref{proofth1} is devoted to the proof of Theorem \ref{th1}, and includes important results on the operator $\calL$ that will be used in the proofs of Theorems \ref{th2} and \ref{th3}, given in Sections \ref{proofth2} and \ref{proofth3}, respectively.

To conclude this introduction, we would like to mention, that after this second work on radiative transfer equations with long-range interactions was completed, we became aware of the recent paper \cite{alonso} that addresses similar questions. Some differences are the following: it seems our hypoelliptic estimates of Theorem \ref{hypo} are sharper; our techniques of proof are different, we use in particular probabilistic techniques for the peaked forward regime instead of PDE techniques; the diffusion limit is not addressed in \cite{alonso}.  

{\bf Acknowledgment.} This work was supported in part by NSF grant
DMS-1311903, an AFOSR NSSEFF Fellowship, and  NSF CAREER grant DMS-1452349.

\section{Preliminaries} \label{prelim}

We introduce here some notations and the main properties of the collision kernel $F$.

 \paragraph{Notation.} We will denote by $\langle \cdot, \cdot \rangle$ the $L^2(\Sm^d)$ inner product. 
The Laplace-Beltrami operator $\Delta_d$ is defined by 
\bee
\Delta_d \varphi(z)&=& \left. \Delta \varphi\left(\frac{y}{\| y\|}\right) \right|_{y= z}, \qquad  z \in \Sm^d\\
&=&\left(\Delta-\sum_{i=1}^{d+1}\sum_{j=1}^{d+1} z_i
  z_j \partial_{z_i}\partial_{z_j}-d
  \sum_{i=1}^{d+1}z_i \partial_{z_i}\right)\varphi(z). 
\eee 
Here,
$\Delta$ is the standard Laplacian on $\Rm^{d+1}$. The eigenfunctions of
$-\Delta_d$ are the spherical harmonics $Y_{n,m}$, for $n \in \Nm$ and
$m=1,\cdots,M(d,n)$ where
$$
M(d,n)=(2n+d-1)\frac{\Gamma(n+d-1)}{\Gamma(d) \Gamma(n+1)},
$$
and are associated with the eigenvalues $\lambda_n=n(n+d-1)$. Above,
$\Gamma$ is the gamma function. We will denote by $e_0=1/(\sigma(\Sm^d))^{1/2}$ the first spherical harmonics $Y_{0,1}$. The Fourier representation of the
fractional Laplace-Beltrami operator is then, for $\beta \in
(0,\infty)$,
$$
(-\Delta_d )^\beta \varphi(k)=\sum_{n=0}^\infty \sum_{m=1}^{M(d,n)} \lambda_n^\beta \; \langle\varphi,Y_{n,m}\rangle Y_{n,m}(k),
$$
where convergence is understood in $L^2(\Sm^d)$.  
The Sobolev space $H^\theta(\Sm^d)$, for~$\theta>0$, is defined by
$$
H^\theta(\Sm^d)=\left\{ \varphi \in L^2(\Sm^d),
\quad (-\Delta_d )^\frac{\theta}{2} \varphi \in L^2(\Sm^d)\right\},
$$
and is equipped with the norm
$$
\|\varphi\|^2_{H^\theta(\Sm^d)}=\|\varphi\|^2_{L^2(\Sm^d)}+\|(-\Delta_d )^\frac{\theta}{2} \varphi\|^2_{L^2(\Sm^d)}.
$$
We will denote by $\calC^p_c$ the space of $\calC^p$ functions with compact support, and by $\calC^p_b$ the space of bounded $\calC^p$ functions. We will use the following convention for the Fourier transform in $\Rm^{d+1}$:
$$
\hat f(\xi)=\calF f(\xi)=\int_{\Rm^{d+1}} e^{- i x \cdot \xi} f(x) dx, \qquad \calF^{-1} f(x)=\frac{1}{(2\pi)^{d+1}} \int_{\Rm^{d+1}} e^{ i x \cdot \xi} \hat f(\xi) d\xi,
$$ 
and introduce the fractional derivative as 
\[
\partial_{x_j}^\gamma f(x)=\calF^{-1} [(i \xi_j)^\gamma \hat f(\xi)](x),
\]
with a similar definition for fractional derivatives involving the
time variable. We will denote by $\partial_{t,x}^\gamma f$ any of the fractional derivatives with respect to $t$ or $x_j$.

\paragraph{The collision kernel.} We suppose $F$ is non-negative and satisfies the following hypotheses: 
$$
\forall \delta \in (-1,1), \qquad F(s) (1-s^2)^\frac{d-2}{2} \in L^1((-1,\delta )),$$
 and there exists $a_1 \in (0,\infty)$ such that
\be \label{hypker}
\lim_{s \to 1}|1-s|^{\beta+\frac{d}{2}}F(s)=a_1, \qquad \beta \in (0,1).
\ee 
We will use the following decomposition of the kernel into smooth and singular parts: let first
$$ 
a_2(s)=\frac{1}{a_1}(|1-s|^{\beta+\frac{d}{2}}F(s)-a_1).
$$
According to \fref{hypker}, $a_2(s) \to 0$ as $s \to 1$, and therefore there exists $\delta \in (0,1)$ such that $|a_2(s)|<1$ for $s \in (\delta,1)$. Let then $\chi \in \calC^\infty([-1,1])$ such that $\chi(s)=1$ for $s \in [-1,\delta]$, and $\chi(s)=0$ for $s \in [\delta',1]$ with $\delta' \in (\delta,1)$. The kernel is finally written as
\be \label{decompF}
F(s)= \chi(s) F(s)+(1-\chi(s))\left(\frac{a_1(1+a_2(s))}{|1-s|^{\beta+\frac{d}{2}}}\right):=F_1(s)+F_2(s).
\ee
We have by construction $F_1 \in L^1((-1,1))$. 

Note that the so-called mean free path, defined as the inverse of the integral of $F$, is equal to zero since 
$$
\int_{\Sm^d} F( k \cdot  p) d\sigma( p) \sim \int_{\Sm^d} \frac{d\sigma( p) }{|1- k \cdot  p|^{\beta+\frac{d}{2}}} \sim \int^1_{-1}  \frac{(1-t^2)^\frac{d-2}{2}}{|1-t|^{\beta+\frac{d}{2}}} dt =\infty.
$$

For $f \in \calC^\infty(\Sm^d)$ (such a regularity is in fact not needed), the collision operator is defined more rigorously by
$$
\calL f( k)= \textrm{p.v.} \int_{\Sm^d} F( k \cdot  p) \left(f( p)-f( k)\right)d\sigma(p),
$$
where p.v. stands for the  principal value. It is shown in Lemma \ref{lemL} that $\calL f \in L^\infty(\Sm^d)$ for such $f$.\\

\section{Main results}\label{res}
We present two types of result: our first result concerns the regularity of the solutions to \fref{transportL}, and our other results address the asymptotic behavior of the solutions in two different regimes, the diffusion regime, and the peaked forward regime. We start with the regularity results.

\paragraph{Hypoelliptic estimates and regularity.} For any $f_0 \in L^2(\Rm^{d+1} \times \Sm^d)$, we say that $ f\in L^\infty((0,\infty),L^2(\Rm^{d+1} \times \Sm^d))$ is a weak solution to \fref{transportL}, if, for all $\varphi \in \calC^\infty([0,\infty) \times \Rm^{d+1} \times \Sm^d)$ with compact support in $x$, with in addition $\varphi=0$ for $t\geq T$, $T$ arbitrary,
\begin{align*}
&\int_0^\infty \int_{\Rm^{d+1}} \int_{\Sm^d} f(t,x, k) (\partial_t+  k \cdot \nabla_x+\calL ) \varphi(t,x, k) dt dx d\sigma( k)\\
&=-\int_{\Rm^{d+1}} \int_{\Sm^d} f^0(x, k) \varphi(0,x, k) dx d\sigma( k).
\end{align*}
Note that the definition makes sense since $\calL \varphi \in L^\infty(\Sm^d)$ when $\varphi \in \calC^\infty(\Sm^d)$ according to Lemma \ref{lemL}.
Our first theorem is the following:

\begin{theorem}(Regularity) \label{th1}The operator $\partial_t+ k \cdot \nabla_x-\calL$ is hypoelliptic. Namely, for any $f^0 \in L^2(\Rm^{d+1} \times \Sm^d)$, \fref{transportL} admits a unique weak solution $f$ that satisfies
$$
 f\in \calC^\infty((0,\infty) \times \Rm^{d+1} \times \Sm^d).
$$
\end{theorem}

The proof of Theorem \ref{th1} is given in Section \ref{proofth1}. It is based on a regularization procedure in order to obtain, along with uniform estimates, existence and uniqueness of weak solutions. We then use Theorem \ref{hypo} in order to gain better regularity in the $(t,x)$ variables first, and show in a second step the improved regularity in the $k$ variable.

\paragraph{Probabilistic representation and asymptotics.} We give below a probabilistic interpretation of the solutions to $\fref{transportL}$, and perform the asymptotic analysis in the probabilistic framework. The starting point is the fact, proved in Lemma \ref{lemL} further, that the operator $\calL$ defined on $\calC^\infty$ can be extended to a unique non-positive self-adjoint operator  $\bar \calL$ with domain $\calD(\bar \calL)=\{\varphi \in H^\beta(\Sm^d), \bar \calL \varphi \in L^2(\Sm^d)\}$. Moreover, $-\bar \calL$ is associated to a quadratic form $\calQ: H^\beta(\Sm^d) \times H^\beta(\Sm^d) \to \Rm_+ $ given by
$$
\calQ(f,g)=\frac{1}{2}\int_{\Sm^d}\int_{\Sm^d} F(k \cdot p)(f(k)-f(p))(g(k)-g(p))d\sigma(k) \sigma( p).
$$
Then, according to \cite[Theorem 7.2.1 pp. 302]{fuku}, there exists a Markov process 
\[\mathbf{M}=(\Omega,\calF,(m(t))_{t\geq0},(\mathbb{P}_{k})_{k\in\Sm^d})\]
 on $\Sm^d$, c\`adl\`ag (right continuous with left limits), with generator $\bar \calL$, and where $(\Omega,\calF)$ is a measurable space. The subscript $k \in \Sm^d$ in the measure $\mathbb{P}_{k}$ is the initial condition of the Markov process. For $x\in \Rm^{d+1}$, considering then 
 \[ X_x(t)=\Big(x-\int_0^t m(u) du ,  m(t)\Big),\]
we have a Markov process on $\Rm^{d+1}\times \Sm^d$ with generator 
\[\tilde{\calL}\varphi=-k\cdot \nabla_x \varphi+\bar\calL \varphi, \quad \textrm{whose domain includes }\quad \calC^1(\Rm^{d+1},\calD(\bar \calL)).\]
Introducing $T_t$ the semigroup $T_t f^0(x,k)=\E_k[f^0(X_{x}(t))]$ associated to $\tilde \calL$, where $\E_k$ is the expectation with respect to the measure $\mathbb{P}_{k}$, it follows, for instance from \cite[Proposition 1.5 pp. 9]{ethier}, that for all $f^0\in \calC^1(\Rm^{d+1},\calD(\bar \calL))\cap L^2(\Rm^{d+1}\times \Sm^d)$, the function $u(t,x,k):=T_t f^0(x,k) \in \calC^1(\Rm^{d+1},\calD(\bar \calL))$ satisfies, 
$$
\partial_t u =\tilde{\calL} u, \qquad \textrm{for all} \quad (t,x,k)\in (0,+\infty) \times \Rm^{d+1}\times \Sm^d.
$$
Above and in the sequel, we use the shorthand $f^0(X_{x}(t))$ for $f^0((X_{x}(t))_1,(X_{x}(t))_2 )$. According to Theorem \ref{th1}, the latter equation admits a unique smooth solution, and therefore $u(t,x,k)=f(t,x,k)$. The semigroup $T_t$ can  be extended to $L^2(\Rm^{d+1}\times \Sm^d)$ thanks to the following estimate, proved in Proposition \ref{existweak},
$$
\| f\|_{L^\infty((0,\infty),L^2(\Rm^{d+1}\times \Sm^d)} \leq \| f_0\|_{L^2(\Rm^{d+1}\times \Sm^d)}.
$$
In this probabilistic framework, we therefore have
$$f(t,x,k)=T_t f^0(x,k)=\E_k[f^0(X_{x}(t))],
$$
and we investigate now two asymptotic regimes. The first one is the diffusion regime, where strong collision effects are investigated in the long time limit. There are two equivalent ways to study the limit, either by rescaling the collision operator as $\calL\to \calL /\eps$ and the time variable as $t \to t/\eps$, or by the change of variables (this is the route we follow here)
$$
t\to \frac{t}{\eps^2}, \qquad x\to \frac{x}{\eps}, \qquad  f(t,x,k) \to f^\eps(t,x,k)=f\Big(\frac{t}{\eps^2},\frac{x}{\eps},k\Big),
$$ 
with $f^\eps(0,x,k)=f^0(x,k)$. The transport equation \fref{transportL} then becomes
\be\label{diffusioneq}
\eps^2 \partial_t f^\eps+ \eps  k \cdot \nabla_x f^\eps=\calL f^\eps.
\ee
In this regime, the probabilistic representation is obtained by considering the rescaled Markov process
\[X^\eps_x(t)=\Big(x-\eps\int_0^{t/\eps^2} m(u) du , m(t/\eps^2)\Big),\]
with generator 
\[\tilde{\calL}_\eps\varphi=-\frac{1}{\eps}k\cdot \nabla_x \varphi+\frac{1}{\eps^2}\bar \calL \varphi.\]
The solution to \eqref{diffusioneq}  then reads
$$f^\eps(t,x,k)=T^\eps_tf^0(x,k)=\E_k[f^0(X^\eps_{x}(t))]. 
$$

Our second result is the following:
\begin{theorem} \label{th2} (Diffusion limit) Suppose there exists $\delta>0$ such that $F \geq \delta$ a.e. on $(-1,1)$, and let $Y_{x}^\eps=x-\eps\int_0^{t/\eps^2}m(u) du$. Then, for all $x\in \Rm^{d+1}$, the process $Y_{x}^\eps$ converges in law in $\calC^0([0,\infty), \Rm^{d+1})$ to a diffusion process $Y_x$, starting at $x$, with generator
$$
\tilde{\calL}_0= \nabla_x \cdot D \nabla_x,
$$ 
where the positive-definite diffusion matrix $D$ is given by 
$$
D_{jl}=\frac{1}{\mathfrak{C}}\int_{\Sm^d} d\tilde{\sigma}(k) k_jk_l \qquad\text{with}\qquad \mathfrak{C}=\sigma(\Sm^{d-1})\int_{-1}^1F(s)(1-s^2)^{(d-2)/2}(1-s)ds.
$$
Above, $\tilde{\sigma}$ is the uniform measure on $\Sm^{d}$ (i.e. $\tilde{\sigma}=\sigma/\sigma(\Sm^d)$). Moreover, for any $f^0 \in L^2(\Rm^{d+1}\times \Sm^d)$ and for all $t>0$, $f_\eps(t)$ converges weakly in $L^2(\Rm^{d+1}\times \Sm^d)$ to the unique solution to
\begin{equation}\label{eqdiffusion}
\partial_t f= \nabla_x \cdot \big(D\nabla_x f\big) \qquad\text{with}\qquad f(0,x)=\int_{\Sm^d} d\tilde{\sigma}(p) f^0(x,p),
\end{equation}
and the function $f$ reads
$$
f(t,x)=\E^{Y_x}[f(0,y_t)]
$$
where $ \E^{Y_x}$ denotes expectation with respect to the law of $Y_x$, and $y$ is the canonical map defined by $y_t(\omega)=\omega(t)$ for all $\omega\in \calC^0([0,\infty), \Rm^{d+1})$.
\end{theorem}

Let us give a few comments on this theorem. The hypothesis that $F$ is strictly positive is needed for the spectral gap estimate of Lemma \ref{lemL}. Moreover, the constant $\mathfrak{C}$ is non-zero and finite, and so does $D$ since, according to \fref{hypker}, 
$$
F(s)(1-s^2)^{(d-2)/2}(1-s) \sim (1-s)^{-\beta} \in L^1((-1,1)) \quad \textrm{for} \quad \beta \in (0,1).
$$
Finally, probabilistic techniques yield slightly stronger convergence results than standard techniques that would provide weak-$\star$ convergence in $L^2((0,T) \times \Rm^{d+1}\times \Sm^d)$. Here, convergence is pointwise in $t$ and weak in $L^2$ for $L^2$ initial conditions. Naturally, convergence can be made stronger by introducing the first order corrector, we did not pursue this route here.

The main ingredient in the proof (given in Section \ref{proofth2}) is a spectral gap estimate that shows that the Markov process $Y^\eps_{x}$ is ergodic, which allows us to use standard diffusion-approximation theorems to obtain convergence.

Our last result concerns the peaked forward limit. In this regime, the main contribution to the scattering process is due to grazing collisions. This is translated to the kernel $F$ by supposing that $F(s)$ is small for $s \neq 1$, and that $F(1)$ is large. According to \fref{hypker}, this suggests the scaling
\be \label{equivpeak}
F^\eps(s)=\eps^{\beta+d/2} K(\eps(1-s)),\qquad\text{with}\qquad K(t)\underset{t\to0}{\sim} \frac{a_1}{\vert t \vert^{\beta+d/2}},
\ee
and $(1-s^2)^\frac{d-2}{2} K(1-s) \in L^1((-1,\delta)),$ for all $\delta \in (-1,1)$. The rescaled generator is accordingly $\bar \calL_\eps$, and the corresponding Markov process is now
 \[\mathbf{M}^\eps=(\Omega,\calF,(m^\eps(t))_{t\geq0},(\mathbb{P}_k)_{k\in\Sm^d}).\]
  Consider then 
 \[ X^\eps_{x}(t)=\Big(x-\int_0^t m^\eps(u) du , m^\eps(t)\Big),\]  
with generator 
\[\tilde{\calL}_\eps\varphi=-k\cdot \nabla_x \varphi+\bar\calL_\eps \varphi,\]
and associated semigroup $T^\eps_t$. Then, the function $f^\eps(t,x,k)=T^\eps_tf^0(x,k)=\E_k[f^0(X^\eps_{x}(t))]$ satisfies
\begin{equation}\label{peak}
 \partial_t f^\eps+  k \cdot \nabla_x f^\eps=\bar{\calL}_\eps f^\eps.
\end{equation}
Above, $\E_k$ denotes expectation with respect to the measure $\mathbb{P}_k$. Let us denote by $\mathbb{X}^\eps_{x,k}$ the law of $X^\eps_x$ starting at $(x,k)$. Denoting by $\calD([0,\infty))$ the space of c\`adl\`ag functions equipped with the Skorohod topology \cite{billingsley1}, our third result is the following: 
\begin{theorem} \label{th3}  (Peaked forward limit) For all $(x,k)\in \Rm^{d+1}\times\Sm^d$, the measure $\mathbb{X}^\eps_{x,k}$ converges weakly in $\calD([0,\infty), \Rm^{d+1}\times \Sm^d)$ to a mesure $\mathbb{X}_{x,k}$, which is the law of a diffusion process starting at $(x,k)$ with generator
$$
\tilde{\calL}_0= -k \cdot \nabla_x +\calL_\beta,
$$ 
where 
$$
\calL_\beta \varphi(k)=\textrm{p.v.} \; a_1\int_{\Sm^d} \frac{\varphi(p)-\varphi(k)}{\vert p- k \vert^{2\beta+d}}  d\sigma(p).
$$
Moreover, for any $f^0 \in L^2(\Rm^{d+1}\times \Sm^d)$ and all $t>0$, $f^\eps(t)$ converges weakly in $L^2(\Rm^{d+1}\times \Sm^d)$ to the unique solution to
\begin{equation} \label{FP}
\partial_t f+ k \cdot \nabla_x f=\calL_\beta f\qquad\text{with}\qquad f(0,x,k)=f^0(x,k),
\end{equation} 
and the function $f$ reads
$$
f(t,x,k)=\E^{\mathbb{X}_{x,k}}[f^0(y_t)],
$$
where $\E^{\mathbb{X}_{x,k}}$ denotes the expectation with respect to $\mathbb{X}_{x,k}$, and $y$ is the canonical map defined by $y_t(\omega)=\omega(t)$ for all $\omega\in \calD([0,\infty), \Rm^{d+1}\times \Sm^d)$. When $f^0 \in L^2(\Rm^{d+1}\times \Sm^d) \cap \calC_b^0(\Rm^{d+1}\times \Sm^d)$, then the convergence holds pointwise in $(t,x,k)$.
\end{theorem}

Note that the operator $\calL_\beta$ is not the fractional Laplace-Beltrami operator. It would be in the Euclidean case if $\Sm^d$ were replaced by $\Rm^{d+1}$. Nevertheless, the high frequency behavior of $\calL_\beta$ is the same as the Laplace-Beltrami operator $(\Delta_d)^{\beta}$ as will be clear in the proof of Lemma \ref{lemL}. As in the diffusion limit, we also observe slightly stronger convergence results with the probabilistic techniques, in particular pointwise convergence in $(t,x,k)$ for bounded and continuous initial conditions. Note as well that the solutions to \fref{FP} are $\calC^\infty$ according to Theorem \ref{th1}.
\section{Proof of Theorem \ref{th1}} \label{proofth1}

The proof is divided into several steps. We first derive some important results on the operator $\calL$. In a second time, we regularize the kernel $F$ in order to show the existence and uniqueness of weak solutions along with some energy estimates. We finally obtain the $\calC^\infty$ regularity using Theorem \ref{th1}, first with respect to the $(t,x)$ variables, and then with respect to the momentum $k$.

\subsection{Step 0: Properties of the operator $\calL$}

\begin{lemma} \label{lemL} The operator $\calL$ defined on $\calC^\infty(\Sm^d)$ satisfies the following properties:
\begin{itemize}
\item[(i)] For any $f \in \calC^\infty(\Sm^d)$, $\calL f \in L^\infty(\Sm^d)$.
\item[(ii)] $\calL$ can be extended to a unique non-positive self-adjoint operator $\bar \calL$ with domain $\calD(\bar \calL)=\{\varphi \in H^\beta(\Sm^d), \bar \calL \varphi \in L^2(\Sm^d)\}$. Moreover, $-\bar \calL$ is associated to a quadratic form $\calQ: H^\beta(\Sm^d) \times H^\beta(\Sm^d) \to \Rm_+ $ given by
$$
\calQ(f,g)=\frac{1}{2}\int_{\Sm^d}\int_{\Sm^d} F(k \cdot p)(f(k)-f(p))(g(k)-g(p))d\sigma(k) \sigma( p),
$$
which satisfies the estimate, for some $C>0$,
\be \label{boundQ}
C \|f\|^2_{H^\beta(\Sm^d)} \leq \calQ(f,f)+\|f\|^2_{L^2(\Sm^d)} \leq C^{-1} \|f\|^2_{H^\beta(\Sm^d)}, \qquad \forall f \in H^\beta(\Sm^d).
\ee
Also, $\calL$ can be extended to a bounded operator (still denoted by $\calL$ for simplicity) from $H^\beta(\Sm^d)$ to its dual  $(H^\beta(\Sm^d))^*$.
\item[(iii)] Let $(\cdot )_{-}$ denote the negative part of a function. Assume there exists $b \in (0,a_1)$ such that the kernel $F$ verifies
$$ (F(s)-a_1\vert 1-s\vert ^{-\beta-d/2})_{-} \leq b \vert 1-s\vert ^{-\beta-d/2}, \qquad a.e. \textrm{ on } (-1,1).
$$
Then, there exists a constant $C>0$ such that, for $e_0=\sigma (\Sm^d )^{-1/2}$,
\be \label{boundQ2}
C\|f-\langle f, e_0 \rangle \|^2_{H^\beta(\Sm^d)} \leq \calQ(f,f).
\ee
\item[(iv)] For any $f \in H^{\beta}(\Sm^d)$, there exists $h \in L^2(\Sm^d)$ such that
$$
(-\Delta_d)^{\beta/2} h= \calL f, \qquad \textrm{with} \qquad \| h\|_{L^2(\Sm^d)} \leq C \| f\|_{H^{\beta}(\Sm^d)}.
$$
\end{itemize}
\end{lemma}
\begin{proof} $(i)$ Let $f\in \calC^\infty(\Sm^d)$. We have, for all $k \in \Sm^d$,
$$
\calL f(k)=\lim_{\eta \to 0} \int_{|k-p|>\sqrt{2 \eta}} F(k \cdot p)(f(p)-f(k)) d \sigma(p).
$$ 
We  split $\calL$ into $\calL=\calL_1+\calL_2$ according to \fref{decompF}, and only treat $\calL_2$ since the term $\calL_1$ is straightforward. We may assume without loss of generality that $k=e_{d+1}=(0,\dots,0,1)\in \Rm^{d+1}$, and write, when $d \geq 2$,
\[
p = (\sqrt{1-s^2}u,s),
\]
with $s\in[-1,1]$, and $ u\in \Sm^{d-1}$. Then,
\be \label{decL}
\calL_2 f(k)=\lim_{\eta \to 0} \int_{\Sm^{d-1}} \int_{-1}^{1-\eta} F_2(s) (f(\sqrt{1-s^2}u+s k)-f(k)) (1-s^2)^\frac{d-2}{2}d\sigma(u) ds.
\ee
 Recasting $f$ as
$$
f(k)=\varphi(0,\cdots,0,1), \qquad f(\sqrt{1-s^2}u+s k)=\varphi(\sqrt{1-s^2}u_1, \cdots,\sqrt{1-s^2}u_{d},s),
$$
we have
\bee
f(\sqrt{1-s^2}u+s k)-f(k)&=&(s-1) \partial_{x_{d+1}} \varphi(0,\cdots,0,1)+ \sqrt{1-s^2} u \cdot \nabla_{d}\varphi(0,\cdots,0,1)\\
&&+\calO(|s-1|),
\eee
where $\partial_{x_{d+1}}$ denotes partial derivation with respect to the $(d+1)-$th variable and $\nabla_{d}$ the gradient with respect to the $d$ first variables. Since $\int_{\Sm^{d-1}} u d\sigma(u)=0$, it follows from the equation above and \fref{decompF} that
\be
|\calL_2 f(k)| \leq  C \int_{-1}^{1} \frac{(1-s)(1-s^2)^\frac{d-2}{2}}{(1-s)^{\beta+\frac{d}{2}}} ds \leq C \int_{-1}^{1} \frac{ds}{(1-s)^\beta} \leq C,
\ee
since $\beta \in (0,1)$. This proves item (i) for $d \geq 2$. The case $d =1$ follows analogously after a simple adaptation.\\

$(ii)$ Let $f$ and $g$ be in $\calC^\infty(\Sm^d)$. Then
\begin{equation}\label{relationLQ} \begin{split}
\langle \calL f, g\rangle&= \textrm{p.v.} \int_{\Sm^d}\int_{\Sm^d} F( k \cdot  p) \left(f( p)-f( k)\right) g(k) d\sigma(p) d\sigma(k)\\
&= \frac{1}{2}\int_{\Sm^d}\int_{\Sm^d} F( k \cdot  p) \left(f( p)-f( k)\right) \left(g(k)-g(p)\right) d\sigma(p) d\sigma(k)\\
&=-\calQ(f,g).
\end{split}
\end{equation}
Using \fref{decompF} and a Cauchy-Schwarz inequality, we find
$$
|\langle \calL_2 f, g\rangle|^2 \leq C \left( \int_{\Sm^d}\int_{\Sm^d} \frac{|f( k)-f( p)|^2}{|k- p|^{2\beta+d}} d\sigma( k) \sigma( p) \right)  \left( \int_{\Sm^d}\int_{\Sm^d} \frac{|g( k)-g( p)|^2}{|k-p|^{2\beta+d}} d\sigma( k) d\sigma( p) \right).
$$
Introducing the following operator, for $\beta \in(0,1)$,
$$\calR_\beta f(k)= \textrm{p.v.} \int_{\Sm^d}
\frac{f(k)-f(p)}{|k -p|^{2\beta+d}} d
\sigma(p),~~~k\in\Sm^d,
$$
 we have
$$
\int_{\Sm^d}\int_{\Sm^d} \frac{|f( k)-f( p)|^2}{|k- p |^{2\beta+d}} d\sigma( k) \sigma( p)=2 \langle \calR_\beta f,f \rangle.
$$
It is proved in \cite{samko2} (with a slight adaptation of the
constants), that the Fourier multipliers $R_n$ associated with $\calR_\beta$ are  
$$
R_n=\frac{2^{2\beta} \pi^\frac{d}{2} \Gamma(\beta)}{\Gamma(\frac{d}{2}+\beta)}\left(\frac{\Gamma(n+\frac{d+2\beta}{2})}{\Gamma(n+\frac{d-2\beta}{2})} -\frac{\Gamma(\frac{d+2\beta}{2})}{\Gamma(\frac{d-2\beta}{2})}\right),
$$
where $\Gamma$ is the gamma function. When $d=1$ and
$\beta={1}/{2}$, we have, by convention,~$\Gamma(\frac{d-2\beta}{2})=\Gamma(0)=\infty$. The fact that
$\Gamma(n+\alpha) \sim \Gamma(n) n^\alpha$ as $n \to \infty$ for
$\alpha \in \Rm$, shows that~$R_n$ behaves like $n^{2\beta}$ for large
$n$, which is the same asymptotics as the eigenvalues of
$(-\Delta_d)^{\beta}$. We can therefore write
$$
\langle \calR_\beta f,f \rangle \leq C \|f\|^2_{H^\beta(\Sm^d)},
$$
so that
\be \label{contL}
|\langle \calL f, g\rangle| \leq C  \|f\|_{H^\beta(\Sm^d)}  \|g\|_{H^\beta(\Sm^d)}.
\ee
By density of $\calC^\infty(\Sm^d)$ in $H^\beta(\Sm^d)$,  Theorem X.23 of \cite{RS-80-2} then shows that the quadratic form $\calQ$ can be extended to a quadratic form (still denoted by $\calQ$) with domain $\calD(\calQ)=H^\beta(\Sm^d)$, associated to a unique self-adjoint operator $\bar \calL$, with the domain given in (ii), such that
$$
\calQ(f,g)=-\langle \bar \calL f, g \rangle, \qquad \forall f\in \calD(\bar \calL), \quad \forall g\in H^\beta(\Sm^d).
$$ 
The upper bound in \fref{boundQ} follows from \fref{contL}. For the lower bound, we simply remark that there are positive constants $C_1$ and $C_2$ such that
$$
\langle \calR_\beta f,f \rangle \leq C_1 \langle \calL_2 f, f\rangle+C_2\|f\|^2_{L^2(\Sm^d)}
$$
and use the asymptotic behavior of $R_n$. Note that we can also conclude from \fref{contL} that $\calL$ can be extended to a bounded operator (still denoted by $\calL$) from $H^\beta(\Sm^d)$ to its dual  $(H^\beta(\Sm^d))^*$.\\

$(iii)$ For $f \in H^\beta(\Sm^d)$, we have
\bee
2 \calQ(f,f) 
&=& 
\int_{\Sm^d}\int_{\Sm^d} \left(\frac{a_1}{|k-p|^{2\beta+d}}+F(k \cdot p)-\frac{a_1}{\vert k-p \vert^{2\beta+d}}\right)(f(k)-f(p))^2d\sigma(k) \sigma( p)\\
&\geq& 
\int_{\Sm^d}\int_{\Sm^d} \left(\frac{a_1}{|k-p|^{2\beta+d}}-\left(F(k \cdot p)-\frac{a_1}{\vert k-p \vert^{2\beta+d}}\right)_{-}\right)(f(k)-f(p))^2d\sigma(k) \sigma( p)\\
&\geq& \int_{\Sm^d}\int_{\Sm^d} \frac{a_1-b}{|k -p|^{2\beta+d}}(f(k)-f(p))^2d\sigma(k) \sigma( p)\\
&\geq & 2(a_1-b)\langle \calR_\beta f,f \rangle.
\eee
The conclusion follows once more from the asymptotics of the multipliers $R_n$, and the fact that $R_0=0$.\\

$(iv)$ Let $f \in \calC^\infty(\Sm^d)$. Let us remark first that if $(-\Delta_d)^\gamma: H^{2 \gamma}(\Sm^d) \to L^2(\Sm^d)$, with $\gamma>0$, then
$$
\ker (-\Delta_d)^\gamma=\Cm.
$$
Hence, $(-\Delta_d)^{-\gamma}$, the inverse of $(-\Delta_d)^{\gamma}$, is defined on the set of functions $\varphi \in L^2(\Sm^d)$ with vanishing integral on the sphere, and $\int_{\Sm^d}(-\Delta_d)^{-\gamma} \varphi(k)  d\sigma(k)=0$. Then, since $\int_{\Sm^d} \calL f(k) d\sigma(k)=0$, we can write 
$$
\calL f= (-\Delta_d)^\frac{\beta}{2} (-\Delta_d)^{-\frac{\beta}{2}} \calL f.
$$
It is then direct to see that the operator $(-\Delta_d)^{-\frac{\beta}{2}} \calL$ can be extended to a bounded operator $A$ from $H^\beta(\Sm^d)$ to $L^2(\Sm^d)$. Denote indeed by $\langle g \rangle$ the integral of $g\in \calC^\infty(\Sm^d)$ on $\Sm^d$, so that
$$
\langle (-\Delta_d)^{-\frac{\beta}{2}} \calL f, g \rangle=\langle (-\Delta_d)^{-\frac{\beta}{2}} \calL f, g-\langle g \rangle \rangle=\langle \calL f, (-\Delta_d)^{-\frac{\beta}{2}} (g-\langle g \rangle)\rangle,
$$
we then deduce from \fref{contL} that
$$
|\langle \calL f, (-\Delta_d)^{-\frac{\beta}{2}} (g-\langle g \rangle) \rangle| \leq C  \|f\|_{H^\beta(\Sm^d)}  \|(-\Delta_d)^{-\frac{\beta}{2}} (g-\langle g \rangle)\|_{H^\beta(\Sm^d)} \leq C \|f\|_{H^\beta(\Sm^d)} \|g\|_{L^2(\Sm^d)}.
$$
It finally suffices to set $h=A f$ to conclude the proof.
\end{proof}
\subsection{Step 1: standard estimates for the Cauchy problem}
\begin{proposition} \label{existweak} The radiative transfer equation \fref{transportL} admits a unique weak solution $f$ that satisfies $(-\Delta_d)^\frac{\beta}{2} f \in L^2((0,\infty) \times \Rm^{d+1} \times \Sm^d)$ and the following estimate: for any $T>0$,
\be \label{est_rte1}
\| f\|^2_{L^\infty((0,T),L^2(\Rm^{d+1} \times \Sm^d))}+2\int_0^T \calQ(f,f) dt \leq \|f_0\|^2_{L^2(\Rm^{d+1} \times \Sm^d)},
\ee
where the quadratic form $\calQ$ is defined in Lemma \ref{lemL}. Moreover, if for any $\calC^\infty$ function $\chi$ with compact support in $(0,\infty)$, we have $\partial_{t,x}^{\gamma} {f_{\chi}} \in L^2(\Rm_t \times \Rm_x^{d+1} \times \Sm^d)$,
where $\gamma>0$ and $f_\chi$ is extension by zero of $\chi f$ to all $t$ in $\Rm$, then
\be \label{estimDD}
(-\Delta_d)^\frac{\beta}{2} \partial_{t,x}^{\gamma} {f_{\chi}} \in L^2(\Rm_t \times \Rm_x^{d+1} \times \Sm^d).
\ee
\end{proposition}

\begin{proof}
We consider a slightly more general version of \fref{transportL} that will be helpful for the adjoint problem and for \fref{estimDD}: we suppose there is a source term in the transport equation so that $f$ satisfies
\be \label{rteproof}
\partial_t f+  k \cdot \nabla_x f=\calL f+S, \qquad S \in L^2((0, \infty) \times \Rm^{d+1} \times \Sm^d).
\ee
We then proceed by regularization: let $f^0_n \in \calC^\infty(\Rm^{d+1} \times \Sm^d)$
 with compact support in $x$, $S_n \in \calC^\infty([0,\infty)\times \Rm^{d+1} \times \Sm^d)$  with compact support in $(t,x)$ such that, as $n \to \infty$,
\begin{align*}
&f^0_n \to f^0, \qquad \textrm{strongly in} \qquad L^2(\Rm^{d+1} \times \Sm^d)\\
&S_n \to S, \qquad \textrm{strongly in} \qquad L^2((0,\infty)\times \Rm^{d+1} \times \Sm^d).
\end{align*}
The regularization of $F$ is done as follows: consider decomposition \fref{decompF}. The functions $F_1$ and $a_2$ are regularized into $\calC^\infty([-1,1])$ functions with usual mollifiers as $F_{1,n}$ and $a_{2,n}$ with
\begin{align}
&F_{1,n} \to F_1, \qquad \textrm{strongly in} \qquad L^1((-1,1))\\
&a_{2,n} \to a_2, \qquad \textrm{a.e. on} \qquad (\delta,1). \label{lima}
\end{align}
The function $F_2$ is regularized as
$$
F_{2,n}(s)=\frac{a_1(1+a_{2,n}(s))(1-\chi(s))}{|1-s+n^{-1}|^{\beta+\frac{d}{2}}},
$$
with $1+a_{2,n}(s) \geq 0$ in $[\delta,1]$ since $a_{2,n}$ is  such that $\|a_{2,n}\|_{L^\infty((\delta,1))}\leq \|a_{2}\|_{L^\infty((\delta,1))}<1$. It is then direct to obtain the existence of a smooth solution $f_n \in \calC^\infty([0, \infty) \times \Rm^{d+1} \times \Sm^d)$ with compact support in $x$ (using for instance Duhamel expansions and the fact that the collision kernel is bounded). Multiplying the regularized transport equation by $f_n$ and integrating in all variables leads to, for any $T \in (0,\infty)$,
\begin{align} \label{estimreg}
&\| f_n\|^2_{L^\infty((0,T),L^2(\Rm^{d+1} \times \Sm^d))}+\sum_{i=1}^2\| h_{i,n}\|^2_{L^2((0,T) \times \Rm^{d+1} \times \Sm^d \times \Sm^d)}\\
& \hspace{2cm}\leq \|f_n^0\|^2_{L^2(\Rm^{d+1} \times \Sm^d)}+2\|S_n\|_{L^2((0,T) \times \Rm^{d+1} \times \Sm^d)}\|f_n\|_{L^2((0,T) \times \Rm^{d+1} \times \Sm^d)}\nonumber
\end{align}
where
\be \label{defh}
h_{i,n}(t,x,k,p)=\sqrt{F_{i,n}(k \cdot p)} (f_n(t,x,k)-f_n(t,x,p)).
\ee
Since by construction,
$$ \|f_n^0\|_{L^2(\Rm^{d+1} \times \Sm^d)} \leq  \|f^0\|_{L^2(\Rm^{d+1} \times \Sm^d)}, \qquad \|S_n\|_{L^2((0,\infty) \times \Rm^{d+1} \times \Sm^d)} \leq  \|S\|_{L^2((0,\infty)\times \Rm^{d+1} \times \Sm^d)},$$
 we can deduce uniform bounds from \fref{estimreg}, and standard compactness arguments show that one can extract subsequences such that, for $i=1,2$:
\begin{align} \label{lim1}
&f_n \to f \qquad \textrm{weak-$\star$ in} \qquad L^\infty((0,\infty),L^2(\Rm^{d+1} \times \Sm^d))\\  \label{lim2}
&h_{i,n} \to h_i \qquad \textrm{weakly in} \qquad L^2((0,\infty)\times \Rm^{d+1} \times \Sm^d\times \Sm^d).
\end{align}
It remains now to identify the limits. We focus only on the most interesting terms $f$ and $h_2$ and leave the details for $h_1$ to the reader. We start with $h_{2,n}$. Let $\varphi \in L^2((0,\infty)\times \Rm^{d+1} \times \Sm^d\times \Sm^d)$ and denote by $(\cdot, \cdot)$ the corresponding inner product. Functions of the form $\varphi_1(t) \otimes \varphi_2(x) \otimes \varphi_3(k) \otimes \varphi_4(p)$, with $\varphi_i$ infinitely differentiable with compact support, are dense in $L^2((0,\infty)\times \Rm^{d+1} \times \Sm^d\times \Sm^d)$, see \cite{treves}, Chapter 39, so it is enough to consider a test function $\varphi$ of the latter tensor form. We then have
$$
(h_{2,n},\varphi)=\int_0^T\int_{\Rm^{d+1}} \int_{\Sm^d} \int_{\Sm^d} f_{n}(t,x,k)\varphi_1(t)  \varphi_2(x)  \psi_n(k,p) dt dx d\sigma(k)d\sigma(p),
$$
where
$$
\psi_n(k,p)=\sqrt{F_{2,n}(k \cdot p)} (\varphi_3(k)\varphi_4(p)-\varphi_3(p)\varphi_4(k)).
$$
Owing the weak convergence \fref{lim1}, we need to show the strong convergence of $\psi_n$ in $L^2$ in order to identify the limit. Since $\psi_n$ can be written as
$$
\psi_n(k,p)=\sqrt{F_{2,n}(k \cdot p)} \left(\varphi_3(k)\varphi_4(k)-\varphi_3(p)\varphi_4(p)-(\varphi_3(k)+\varphi_4(p))(\varphi_4(k)-\varphi_4(p))\right),
$$
it is enough to show the strong convergence in $L^2(\Sm^d \times \Sm^d )$ of functions of the form
$$
\widetilde{\psi_n}(k,p)=\sqrt{F_{2,n}(k \cdot p)} (\varphi(k)-\varphi(p))
$$ 
where $\varphi \in \calC^\infty(\Sm^d)$. For this, we remark first that by construction, using \fref{lima}, as $n \to \infty$,
\be \label{point}
{F_{2,n}(s)} \to {F_{2}(s)} \quad \textrm{a.e. on} \quad (\delta,1). 
\ee
Introduce then
$$
I_n:=\int_{\Sm^d}\int_{\Sm^d} \left(\sqrt{F_{2,n}(k \cdot p)}-\sqrt{F_{2}(k \cdot p)}\right)^2(\varphi(k)-\varphi(p))^2 d\sigma(k) d\sigma(p).
$$
Since
$$
|\varphi(k)-\varphi(p)| \leq C |k-p|,
$$
and $F_{2,n}$, $F_{2}$ are supported on $[\delta,1]$, we have
$$
I_n \leq C \int_{\delta}^1 \left(\frac{\sqrt{1+a_{2}(s)}}{|1-s|^{\frac{\beta}{2}+\frac{d}{4}}}-\frac{\sqrt{1+a_{2,n}(s)}}{|1-s+n^{-1}|^{\frac{\beta}{2}+\frac{d}{4}}}\right)^2 (1-s) (1-s^2)^\frac{d-2}{2} ds.
$$
The integrand being uniformly bounded by the function $C(1-t)^{-\beta} \in L^1((-1,1))$ since $\beta \in (0,1)$, \fref{point} together with dominated convergence show that $I_n \to 0$. This implies that
$$
(h_{2,n},\varphi) \to \int_0^T\int_{\Rm^{d+1}} \int_{\Sm^d} \int_{\Sm^d} f(t,x,k)\varphi_1(t)  \varphi_2(x) \psi(k,p) dt dx d\sigma(k)d\sigma(p),
$$
with obvious notation for $\psi$, which yields
$$
h_{2}(t,x,k,p)=\sqrt{F_{2}(k \cdot p)} (f(t,x,k)-f(t,x,p)).
$$
>From \fref{estimreg}, \fref{lim1} and \fref{lim2}, we then deduce \fref{est_rte1} when $S=0$. It only remains to pass to the limit in the weak formulation. The only term requiring some attention is the one involving the collision kernel, for which we need to show that terms of the form
$$
\textrm{p.v.} \int_{\Sm^d} (F_{1,n}+F_{2,n}-F_{1}-F_{2})( k \cdot  p) \left(\varphi( p)-\varphi( k)\right)d\sigma(p)
$$
converge to zero strongly in $L^2(\Sm^d)$ for $\varphi$ smooth. This is done using \fref{point} and the decomposition \fref{decL} in order to get a majorizing function for dominated convergence. Summarizing, we have therefore obtained the existence of a weak solution $f$ satisfying estimate \fref{est_rte1}. The latter, together with \fref{boundQ2}, yields $(-\Delta_d)^\frac{\beta}{2} f \in L^2((0,\infty) \times \Rm^{d+1} \times \Sm^d)$.

Let us prove now the uniqueness. We proceed as usual with the adjoint problem, and need to show that for $f$ a weak solution with a vanishing initial condition, the following equality, for all smooth $\varphi$ with compact support in $x$ such that $\varphi=0$ for $t \geq T$, $T$ arbitrary,
\begin{align*}
&\int_0^T \int_{\Rm^{d+1}} \int_{\Sm^d} f(t,x, k) (\partial_t+  k \cdot \nabla_x+\calL ) \varphi(t,x, k) dt dx d\sigma( k)=0,
\end{align*}
implies $f=0$. Note that for technical reasons, the latter condition can be recast as 
\begin{align} \label{altweak}
&\lim_{n \to \infty}\int_{n^{-1}}^T \int_{\Rm^{d+1}} \int_{\Sm^d} f(t,x, k) (\partial_t+  k \cdot \nabla_x+\calL ) \varphi(t,x, k) dt dx d\sigma( k)=0.
\end{align}
Let $v \in L^\infty((0,T), L^2(\Rm^{d+1} \times \Sm^d))$ be a weak solution to the problem
$$
(\partial_t+  k \cdot \nabla_x-\calL) v(t,x,k)=-f(T-t,x,-k), \qquad v(0,x,k)=0,
$$
and define $u(t,x,k)=v(T-t,x,-k)$. We then verify that $u$ is a weak solution to the adjoint problem
$$
(\partial_t+  k \cdot \nabla_x+\calL) u(t,x,k)=f(t,x,k), \qquad u(T,x,k)=0,
$$
that is, for all $\varphi$ smooth with $\varphi(0,x,k)=0$,
\begin{align} \label{adjoint}
&\int_0^T \int_{\Rm^{d+1}} \int_{\Sm^d} u(t,x, k) (-\partial_t-k \cdot \nabla_x+\calL ) \varphi(t,x, k) dt dx d\sigma( k)\\
&\hspace{1cm}=\int_0^T \int_{\Rm^{d+1}} \int_{\Sm^d} f(t,x, k) \varphi(t,x, k) dt dx d\sigma( k). \nonumber
\end{align}
If $u$ were smooth, we would plug $\varphi=u$ in \fref{altweak}, which would lead to $f=0$. We only know that $u \in L^\infty((0,T), L^2(\Rm^{d+1} \times \Sm^d))$ and thus need to regularize. Let $u_n$ be a regularized version of $u$ defined by, for $t\in [n^{-1},T]$,
$$
u_n(t,x,k)=\int_0^T \int_{\Rm^{d+1}} \int_{\Sm^d}\varphi_n(t-s) \phi_n(x-y) \psi_n(k \cdot p) u(s,y,p) ds dx d\sigma(p),
$$
where $\varphi_n$ is such that $\varphi_n(u)=0$ for $|u|\geq n^{-1}$. We have consequently,
\begin{align*}
&(\partial_t+  k \cdot \nabla_x+\calL ) u_n(t,x,k)\\
& \hspace{1cm} =\int_0^T \int_{\Rm^{d+1}} \int_{\Sm^d}  u(s,y,p) (-\partial_s-k \cdot \nabla_y+\calL )\varphi_n(t-s) \phi_n(x-y) \psi_n(k \cdot p) ds dx d\sigma(p).
\end{align*}
Since $\varphi_n(t)=0$ for $t \in [n^{-1},T]$, we can use the fact that $u$ is a weak solution, and therefore \fref{adjoint}, to write, for all $t \in [n^{-1},T]$,
$$
(\partial_t+  k \cdot \nabla_x+\calL ) u_n(t,x,k)=f_n(t,x,k),
$$
where
$$
f_n(t,x,k)=\int_0^T \int_{\Rm^{d+1}} \int_{\Sm^d}  f(s,y,p) \varphi_n(t-s) \phi_n(x-y) \psi_n(k \cdot p) ds dx d\sigma(p).
$$
Assuming the regularization is such that
$$
f_n \to f, \qquad \textrm{strongly in} \qquad L^2((0,T)\times \Rm^{d+1} \times \Sm^d),
$$
it is then straightforward to conclude from \fref{altweak} that
\begin{align*} 
&\lim_{n \to \infty}\int_{n^{-1}}^T \int_{\Rm^{d+1}} \int_{\Sm^d} f(t,x, k) (\partial_t+  k \cdot \nabla_x+\calL ) u_n(t,x, k) dt dx d\sigma( k)\\
&=\lim_{n \to \infty}\int_{n^{-1}}^T \int_{\Rm^{d+1}} \int_{\Sm^d} f(t,x, k) f_n(t,x, k) dt dx d\sigma( k)\\
&=\int_0^T \int_{\Rm^{d+1}} \int_{\Sm^d} (f(t,x, k))^2 dt dx d\sigma( k)=0,
\end{align*}
which proves the uniqueness.
 
Regarding \fref{estimDD}, suppose $\partial_{t,x}^{\gamma} {f_{\chi}} \in L^2(\Rm_t \times \Rm_x^{d+1} \times \Sm^d)$,  and consider the regularized $f_n$ obtained previously. We remark that $u_n=\partial^\gamma_{t,x} f_{n,\chi}$ ($f_{n,\chi}$ the extension by zero of $\chi f_n$ to all $t \in \Rm$) satisfies \fref{rteproof} on $\Rm_t \times \Rm^{d+1}_x \times \Sm^d$ with $f=u_n$ and $S=S_n=-\partial^\gamma_{t,x} f_{n,\chi'}$. Note also that $\|u_n(t)\|_{L^2(\Rm^{d+1}\times \Sm^d)} \to 0$ as $|t| \to \infty$. Multiplying then \fref{rteproof} by $u_n$ and integrating over all variables leads to
\begin{align} \label{estimreg2}
&\frac{1}{2}\sum_{i=1}^2\| h_{i,n}\|^2_{L^2(\Rm^ \times \Rm^{d+1} \times \Sm^d \times \Sm^d)} \leq \|\partial^\gamma_{t,x} f_{n,\chi'}\|_{L^2(\Rm \times \Rm^{d+1} \times \Sm^d)}\|u_n\|_{L^2( \Rm \times \Rm^{d+1} \times \Sm^d)}
\end{align}
where $h_{i,n}$ is as in \fref{defh} with $f_n$ replaced by $u_n$. We showed above that $f_n$ converges weakly to the unique weak solution $f$. It is then not hard to see that the fact that $\partial_{t,x}^{\gamma} {f_{\chi}}$ belongs to $L^2(\Rm \times \Rm^{d+1} \times \Sm^d)$ for any smooth $\chi$ with compact support in $(0,\infty)$ implies that $S_n$ and $u_n$ are uniformly bounded in $L^2(\Rm \times \Rm^{d+1} \times \Sm^d)$. Following along the same lines as in the proof of \fref{est_rte1}, we can finally pass to the limit in \fref{estimreg2} and obtain that
$$
\int_{\Rm} \calQ(u,u)dt \leq  \|\partial_{t,x}^{\gamma} {f_{\chi'}}\|_{L^2(\Rm \times \Rm^{d+1} \times \Sm^d)}\|u\|_{L^2( \Rm \times \Rm^{d+1} \times \Sm^d)}, \qquad u=\partial_{t,x}^{\gamma} {f_{\chi}},
$$
which yields \fref{estimDD} thanks to \fref{boundQ2}.
\end{proof}

\subsection{Step 2: Regularity in $(t,x)$}

\begin{proposition} \label{regx} Let $f$ be the weak solution to \fref{transportL} and let $\gamma=\frac{\beta}{2+3 \beta}$. Then, for any $n \geq 0$, any $m \geq 0$,  and any $\calC^\infty$ function $\chi$ with compact support in $(0,\infty)$, we have
$$
\partial_{t}^{n \gamma}\partial_{x_j}^{m \gamma} {f_{\chi}} \in L^2(\Rm \times \Rm^{d+1},  H^\beta(\Sm^d)), \qquad j=1,\cdots,d+1,
$$
where $f_{\chi}$ is the extension by zero of the function $\chi f$ to all $t \in \Rm$.
\end{proposition}
\begin{proof} We know from Proposition \ref{existweak} that the result is true for $n=m=0$, and prove the proposition by a double induction on $n$ and $m$. Let us show first that $\partial_{t}^{n \gamma} {f_{\chi}} \in L^2(\Rm \times \Rm^{d+1} \times \Sm^d)$.  By construction, the function $f_{\chi} \in L^2(\Rm \times \Rm^{d+1} \times \Sm^d)$ verifies in the distribution sense
$$
\partial_t f_{\chi}+  k \cdot \nabla_x f_{\chi}=\calL f_{\chi}+S, \qquad S=- f_{\chi'}.
$$
We start with the case $n=1$. We know from Proposition \ref{existweak} that $(-\Delta_d)^\frac{\beta}{2} f \in L^2((0,\infty) \times \Rm^{d+1} \times \Sm^d)$, so that $(-\Delta_d)^\frac{\beta}{2} f_{\chi} \in L^2(\Rm \times \Rm^{d+1} \times \Sm^d)$. We have as well $S \in L^2(\Rm \times \Rm^{d+1} \times \Sm^d)$. In order to apply Theorem \ref{hypo} and the hypoelliptic estimates, we use item $(iv)$ of Proposition \ref{lemL} to write $\calL f_{\chi}=(-\Delta_d)^\frac{\beta}{2} h $, for some $h \in L^2(\Rm \times \Rm^{d+1} \times \Sm^d)$. It then follows from Theorem \ref{hypo} (with $\theta \equiv \beta$ and $\beta \equiv \beta/2$) that 
$
\partial_{t}^{\gamma} f_{\chi} \in L^2(\Rm \times \Rm^{d+1} \times \Sm^d),
$
which proves the case $n=1$.

Suppose now the result holds true at the step $n$, $n \geq 1$, so that for any $\chi$ with the properties stated in the proposition, we have 
$
u:=\partial_{t}^{n \gamma} f_{\chi} \in L^2(\Rm \times \Rm^{d+1} \times \Sm^d).
$
Result \fref{estimDD} of Proposition \ref{existweak} then yields  
$$(-\Delta_d)^\frac{\beta}{2} u \in L^2(\Rm \times \Rm^{d+1} \times \Sm^d).$$ 
In order to conclude, we remark that the function $u$ satisfies in the distribution sense
$$
\partial_t u+  k \cdot \nabla_x u=\calL u+S, \qquad S=- \partial_{t,x}^{n \gamma}f_{\chi'}.
$$
Since $\chi' \in \calC^\infty$ with compact support in $(0,\infty)$, we have by the induction hypothesis that $S \in L^2(\Rm \times \Rm^{d+1} \times \Sm^d)$. Applying once more Theorem \ref{hypo}, we find
$
\partial_{t}^{\gamma} u \in L^2(\Rm \times \Rm^{d+1} \times \Sm^d),
$
which proves the result for the step $n+1$ and therefore for all $n$. Following along the same lines shows that $\partial_{x_j}^{m \gamma} {f_{\chi}} \in L^2(\Rm \times \Rm^{d+1} \times \Sm^d)$ for all $m$, which completes the initialization step of the double induction.

The evolution step from $(n,m)$ to  $(n+1,m)$, to $(n,m+1)$ and to $(n+1,m+1)$ is done in a similar fashion as above, by defining first $u$ by $u:=\partial_{t}^{n \gamma} \partial_{x_j}^{m \gamma} f_{\chi} \in L^2(\Rm \times \Rm^{d+1} \times \Sm^d)$ and then by $u:=\partial_{t}^{(n+1) \gamma} \partial_{x_j}^{m \gamma} f_{\chi} \in L^2(\Rm \times \Rm^{d+1} \times \Sm^d)$. This ends the proof.
\end{proof}

\subsection{Step 3: Regularity in $k$ and conclusion}
Take $t_0>0$ and $T>0$ arbitrary with $t_0<T$, and let $\chi \in \calC^\infty(\Rm)$ such that $\chi(t)=1$ for $t \in [t_0,T]$, and $\chi(t)=0$ for $t\leq t_0/2$ and $t \geq 2T$. It follows from Proposition \ref{regx} that, for all $n\geq 0$ and $m \geq 0$,
$$
\partial_{t}^{m}\partial_{x_j}^{n} f \in L^2((t_0, T) \times \Rm^{d+1},H^\beta(\Sm^d)), \qquad j=1,\cdots,d+1.
$$
Since $t_0$ and $T$ are arbitrary, standard Sobolev embeddings then yield
$$
f \in \calC^\infty_{t,x}((0, \infty) \times \Rm^{d+1}, H^\beta(\Sm^d)).
$$
It thus remains to show the regularity in the $k$ variable. Define for this the operator $A:=-\bar \calL+\mathbb{I}$. It is direct to see that $A$ is a sectorial operator that generates a holomorphic semigroup $S(t)$. For $\theta \in (0, \pi/2]$, let indeed $S_\theta=\{ z\in \Cm, |arg(z)|\leq \pi/2-\theta, \Re(z)\geq 0\}$. The spectrum of $A$, denoted by $\sigma(A)$, satisfies $\sigma(A) \subset [1,\infty)$ and therefore $\sigma(A) \subset S_\theta$, for all $ \theta \in (0,\pi/2]$. Hence, for $z \in \Cm \backslash S_{\theta_1}$ with $0<\theta_1<\theta$, we have, 
$$
\|(z+A)^{-1}\|_{L^2(\Sm^d)}=\frac{1}{ \textrm{dist}(z,\sigma(A))} \leq \frac{1}{ \textrm{dist}(z,S_{\theta_1})},
$$
and it suffices to apply Theorem X.25 of \cite{RS-80-2} (note that $A$ is closed since it is self-adjoint). It then follows from \cite{RS-80-2}, Corollary 2, page 252, that, for any $\varphi \in L^2(\Sm^d)$, $S(t) \varphi \in D(A^p)$, for all $t>0$ and $p \in \Nm^*$, with the estimate
\be \label{holo}
\|A^p S(t) \varphi \|_{L^2(\Sm^d)} \leq C t^{-p} \|\varphi\|_{L^2(\Sm^d)}.
\ee
By interpolation, the latter estimate can be extended to all $p$ real and positive. Let then $u:=\partial_{t}^{m}\partial_{x_j}^{n}f $, $j=1,\cdots,d+1$, $n$ and $m$ arbitrary, let $\chi$ defined as before, and let $v:=-\chi\, k \cdot \nabla_x u - u (\chi'-1)$. Since $f$ is smooth in $(t,x)$, $v$ is defined for all $(t,x)$ and a.e. in $\Sm^d$. This allows us to use the following representation formula for $u \chi$,
$$
u(t)\chi(t)=\int_0^t S(t-s) v(s)ds.
$$
We show now that for all $(m,n,j)$, $A^{p/2} u \in L^\infty((0,T),L^2(\Rm^{d+1} \times \Sm^d))$, for all $p \in \Nm^*$. We proceed by induction, and start with the initial step $p=1$. According to \fref{boundQ}, we have, for all $(t,x)$,
$$
\langle A u, u \rangle = \langle A^{1/2} u, A^{1/2} u \rangle \leq C \|u\|^2_{H^\beta(\Sm^d)},
$$
which proves the claim together with Proposition \ref{regx}. Suppose now that for all $(m,n,j)$, $A^{p/2} u \in L^\infty((0,T),L^2(\Rm^{d+1} \times \Sm^d))$. We have
$$
A^{(p+1)/2} u(t)\chi(t)=\int_0^t A^{1/2}S(t-s) A^{p/2}v(s)ds.
$$
Above, we used the fact that $A^{(p+1)/2}$ is closed, and that $S$ and $A^{p/2}$ commute. Using then \fref{holo}, we find, for $0<t \leq T$,
$$
\|A^{(p+1)/2} u(t)\chi(t)\|_{L^2(\Rm^{d+1}\times \Sm^d)} \leq C \| A^{p/2}v\|_{L^\infty((0,T),L^2(\Rm^{d+1}\times \Sm^d))}\int_0^t (t-s)^{-1/2} ds,
$$
which is finite for all $(m,n,j)$ according to the induction hypothesis and proves the claim. We conclude the proof of the proposition by showing that $(-\Delta_d)^{p\beta /2 } u \in L^\infty((0,T),L^2(\Rm^{d+1}\times \Sm^d))$ for all $p \in \Nm^*$. We write for this, using \fref{boundQ}, for all $p \geq 1$,
\be \label{ap}
C \| A^{p} u\|^2_{H^\beta(\Sm^d)} \leq \langle A A^{p}u, A^{p}u \rangle=\langle A^{2p+1} u,u \rangle=\|A^{p+1/2} u\|^2_{L^2(\Sm^d)}.
\ee
Using the Funk-Hekke formula \cite{harmonics}, Chapter 2, page 36, it is not difficult to see that $A$ and $(-\Delta_d)^\frac{\beta}{2}$ commute, so that
 \bee
\| A^{p} u\|^2_{H^\beta(\Sm^d)}&=&\| A^{p} u\|^2_{L^2(\Sm^d)}+\| (-\Delta_d)^{\beta/2 }A^{p} u\|^2_{L^2(\Sm^d)}\\
&=&\| A^{p} u\|^2_{L^2(\Sm^d)}+\| A^{p} (-\Delta_d)^{\beta/2 } u\|^2_{L^2(\Sm^d)}.
\eee
With \fref{ap}, this finally yields
$$
\| A^{p} (-\Delta_d)^{\beta/2 } u\|_{L^2(\Sm^d)}\leq C \|A^{p+1/2} u\|_{L^2(\Sm^d)},
$$
and it suffices to iterate to conclude that, for any $0<t_0<T$, any $(m,n,j,p)$,
$$
(-\Delta_d)^{p \beta/2 }\partial_{t}^{m}\partial_{x_j}^{n}f \in L^\infty((t_0,T),L^2(\Rm^{d+1}\times \Sm^d)).
$$
The proof is ended with classical Sobolev embeddings.

\section{Proof of Theorem \ref{th2}}\label{proofth2}
The proof is based on standard convergence arguments for stochastic processes. The main ingredient is the spectral gap estimate $(iii)$ of Lemma \ref{lemL}.

\paragraph{Step 1: spectral gap and invariant measure.} Let us first verify that the hypothesis on the kernel $F$ in item $(iii)$ is verified. According to \fref{hypker}, for all $b>0$, there exists $\delta_0$ such that
$$ |F(s)-a_1\vert 1-s\vert ^{-\beta-d/2}| \leq b \vert 1-s\vert ^{-\beta-d/2}, \qquad \forall s \in [\delta_0,1).
$$
It remains to treat the case $s \in (-1,\delta_0)$. It is assumed here that there exists $\eta$ such that $F(x)\geq \eta>0$ a.e. on $(-1,1)$. Suppose $F(s)(1-s) ^{\beta+d/2}-a_1 \leq 0$ for a non empty subset $I$ of $(-1,\delta_0)$, otherwise we are done. Choose then $b<a_1$ such that the associated $\delta_0$ yields $a_1-\eta(1-\delta_0)^{\beta+d/2}>0$. Then, for all $s \in I$,
$$
a_1-F(s)(1-s) ^{\beta+d/2} \leq a_1-\eta(1-\delta_0)^{\beta+d/2},
$$
and the hypothesis is verified. The next lemma is a direct consequence of Lemma \ref{lemL} item $(iii)$:

\begin{lemma}\label{gap} We have the following properties for the operator $\bar \calL$ and the Markov process $\mathbf{M}$.
\begin{enumerate}
\item $0$ is a simple eigenvalue of $\bar \calL$ with associated normalized eigenvector the constant function equal to $e_0=\sigma (\Sm^d )^{-1/2}$.
\item There exists $\mathfrak{g}>0$ such that for all $\varphi\in L^2(\Sm^d)$, the semigroup 
\[P_t \varphi (k):=\E_k[\varphi(m(t))]\]
satisfies  for all $t\geq 0$
\begin{equation}\label{specgap}  \| P_t \varphi -\big<\varphi,\mathbf{1}\big>_{L^2(\tilde{\sigma})}\|_{L^2(\Sm^d)}  \leq e^{-\mathfrak{g}t}\|\varphi\|_{L^2(\Sm^d)}\end{equation}
with
\[\tilde{\sigma}=\frac{\sigma}{\sigma(\Sm^d)}\]
the uniform measure on $\Sm^d$.
\item $\tilde{\sigma}$ is the unique invariant measure for the Markov process $\mathbf{M}$.
\end{enumerate}
\end{lemma}
\begin{proof}
The proof of the first item is straightforward owing Lemma \ref{lemL}. Regarding the second item, we first remark that $\tilde \sigma$ is an invariant measure for $\mathbf{M}$. This follows indeed from the fact that 
$$
\int_{\Sm^d} \bar \calL \varphi (k) d \tilde \sigma(k)=0, \qquad \forall \varphi \in \calC^\infty(\Sm^d), 
$$
and \cite[Proposition 9.2 pp. 239]{ethier}. Let now $v(t)=\|P_t\phi\|^2_{L^2(\tilde{\sigma})}$ with $\phi=\varphi- \big<\varphi , \mathbf{1}\big>_{L^2(\tilde{\sigma})}$ and $\varphi \in \calD(\bar \calL)$. Since $\tilde \sigma$ is an invariant measure, $\langle P_t\phi, \mathbf{1}\rangle_{L^2(\tilde{\sigma})}=\langle\phi, \mathbf{1}\rangle_{L^2(\tilde{\sigma})}=0$. Hence, 
\[v'(t)=2\big<P_t\phi,\bar{\calL} P_t \phi\big>_{L^2(\tilde{\sigma})}\leq -\mathfrak{g} \| P_t\phi\|^2_{L^2(\tilde{\sigma})}=-\mathfrak{g} v(t), \]
by application of item $(iii)$ of Lemma \ref{lemL} to $P_t\phi$. Therefore,  $t\mapsto v(t)e^{\mathfrak{g} t }$ is a nonincreasing function, and then
\[v(t) = \|P_t\varphi -  \big<\varphi , \mathbf{1}\big>_{L^2(\tilde{\sigma})} \|^2_{L^2(\tilde{\sigma})}\leq e^{-\mathfrak{g} t} \|\phi\|^2_{L^2(\tilde{\sigma})} \leq e^{-\mathfrak{g} t} \|\varphi\|^2_{L^2(\tilde{\sigma})}.\]
The latter estimate is then extended to all $\varphi \in L^2(\Sm^d)$ by density. For the last point, if $\mu$ is another invariant measure, we have by definition $\int \varphi d\mu=\int P_t\varphi d\mu$ for all $t\geq 0$ and all $\varphi\in\calC^\infty(\Sm^d)$. Because of \eqref{specgap}, there exists $(t_n)_n$ such that $t_n\to+\infty$ and
\[\lim_{n\to+\infty} P_{t_n}\varphi = \int_{\Sm^d}\varphi d\tilde{\sigma}\quad\text{a.e. on } \Sm^d.\]
Dominated convergence then yields
\[  \int_{\Sm^d}\varphi d\mu=\int_{\Sm^d}\varphi d\tilde{\sigma},\]
which shows that $\mu=\tilde \sigma$ since $\varphi$ is arbitrary.
\end{proof}

\paragraph{Step 2:  convergence of $Y_{x}^\eps$.} We apply here standard diffusion-approximation theorems based on the martingale formulation and perturbed test functions, see e.g. \cite{FGPS-07,kunita,Kushner,stroock}. These theorems hold for instance when $m$ is an ergodic Markov process. This is the case here since a consequence of the spectral gap estimate \fref{specgap} is that $m$ has a unique invariant measure (see Lemma \ref{gap} (3)), which implies that $m$ is ergodic. Note that $m$ is not stationary since its initial condition is not drawn according to the invariant measure. This has no consequence since \fref{specgap} shows that $\bar{\calL}$ satisfies the Fredholm alternative, which is what is mostly needed for the construction of the perturbed test functions in the martingale techniques: as soon as $\int_{\Sm^d}d\tilde{\sigma}(k)g(k)=0$, the Poisson equation
$$\bar{\calL}u=g,\qquad g\in L^2(\Sm^d)
$$
admits a unique solution (up to a constant) that reads
\begin{equation}\label{solh}u(k)=-\int_0^{+\infty} P_t g(k)dt.\end{equation}
Here $(P_t)_{t\geq 0}$ is the semi-group associated with the Markov process $\mathbf{M}$. The integral in \eqref{solh} is well defined thanks to \eqref{specgap}.

As a result, we can apply the techniques of \cite[Chapter 6]{FGPS-07} or of \cite[Theorem 6.1]{garnier_esaim}, in order to show that $(Y^\eps_{x})_\eps$ converges in law in $\calC^0([0,\infty),\Rm^{d+1})$ as $\eps\to 0$ to a diffusion process with generator 
\[\tilde{\calL}_0\varphi = \sum_{j,l=1}^{d+1} \partial^2_{x_j x_l}\varphi \int_0^{+\infty}\E_{\tilde{\sigma}}[m_j(0)m_{l}(u)]du,\]
where $\E_{\tilde{\sigma}}$ denotes expectation with respect to the invariant measure $\tilde{\sigma}$, and $m_l(u)$ is the $l$-th component of the vector $m(u)$. The explicit form of the diffusion matrix is given in the lemma below:
\begin{lemma}\label{coefD}
We have
\[D_{jl}:=\int_0^{+\infty}\E_{\tilde{\sigma}}[m_j(0) m_l(u)]du=\frac{1}{\mathfrak{C}}\int_{\Sm^d} d\tilde{\sigma}(k) k_jk_l,\]
with
\[\mathfrak{C}=\sigma(\Sm^{d-1})\int_{-1}^1F(s)(1-s^2)^{(d-2)/2}(1-s)ds \in (0,\infty).\]
\end{lemma}
\begin{proof} First, let us remark that, since $m_j(0)=k_j$,
\[\E_{\tilde{\sigma}}[m_j(0)m_l(u)]=\int_{\Sm^d} d\tilde{\sigma}(k)k_j\E_{\tilde{\sigma}}[m_l(u)]. \]
Then, since $m$ is a Markov process, the process $\varphi_l(m(t))-\varphi_l(m(0))-\int_0^t\bar{\calL}\varphi_l (m(u))du$ with $\varphi_l(p)=p_l$, is a martingale. Hence,
\[\begin{split} &\E_{\tilde{\sigma}}[m_l(t)]=k_l+\E_{\tilde{\sigma}} \Big[ \int_0^t  \textrm{p.v.}\int_{\Sm^d}d\sigma(p)F(p\cdot m(u))(p_l-m_l(u))du\Big] \\
&=k_l+\E_{\tilde{\sigma}} \Big[\int_0^t \lim_{\eta \to 0}\int_{-1}^{1-\eta}ds\int_{\Sm^{d-1}}d\sigma(v)F(s)(1-s^2)^{(d-2)/2}(\sqrt{1-s^2} v_l+(s-1)m_l(u) )du\Big].
\end{split}\]
Since $\int_{\Sm^{d-1}}d\sigma(v) v_l=0$, and $F(s)(1-s^2)^{(d-2)/2}(1-s)$ is integrable thanks to \fref{hypker}, the limit $\eta \to 0$ is finite and we obtain
\[\begin{split} &\E_{\tilde{\sigma}}[m_l(t)]=k_l-\sigma(\Sm^{d-1})\int_{-1}^1ds F(s)(1-s^2)^{(d-2)/2}(1-s)\int_0^t \E_{\tilde{\sigma}}[m_l(u)]du.
\end{split}\]
Therefore, we have $\E_{\tilde{\sigma}}[m_l(t)]=k_l e^{-\mathfrak{C}t}$ and then
\[ D_{jl}=\frac{1}{\mathfrak{C}}\int_{\Sm^d} d\tilde{\sigma}(k) k_jk_l.\]
\end{proof}

\paragraph{Step 3: convergence to the diffusion equation and conclusion.} The following result shows that the limit of $f^\eps$ is characterized by the limit of $Y_{x}^\eps$. 

\begin{lemma}\label{ergodic}
We have for all $t>0$, all $f^0 \in L^2(\Rm^{d+1} \times \Sm^d)$, and all $\varphi \in L^2(\Rm^{d+1}\times \Sm^d) $:
\[ \lim_{\eps\to 0}\int_{\Rm^{d+1}\times \Sm^d} dxd\sigma(k)\left(f^\eps(t,x,k) - \tilde{f}^\eps(t,x,k)\right)\varphi(x,k)=0,\]
where
\[\tilde{f}^\eps(t,x,k) = \int_{\Sm^d} d\tilde{\sigma}(p) \E_k[f^0(Y^\eps_{x}(t),p)],\]
with $\tilde{ \sigma}$ the uniform measure on $\Sm^d$ and we recall that
\[Y^\eps_{x}(t)=x-\eps \int_0^{t/\eps^2}m(u)du.\]
\end{lemma}
\begin{proof}
The proof follows once more from the spectral gap estimate. Following Theorem \ref{th1}, we know that
\[f_\eps(t,x,k)=T_t^\eps f^0(x,k)=\E_k\Big[f^0\Big(x-\eps\int^{t/\eps^2}_0 m(u)du,m(t/\eps^2)\Big)\Big],\]
 as well as $\tilde{f}^\eps$ are smooth functions (note that we defined $T_t^\eps$ on $L^2$ here). Hence, setting
\[\phi(x,k):=f^0(x,k)-\big<f^0(x,\cdot) ,\mathbf{1}\big>_{L^2( \tilde{ \sigma})}, \qquad \psi^\eps(t,\cdot,\cdot):=T^\eps_t\phi \equiv f^\eps (t,\cdot,\cdot)- \tilde{f}^\eps(t,\cdot,\cdot),\]
and defining $v(t)=\|\psi^\eps(t) \|^2_{L^2(\Rm^{d+1}\times\Sm^d)}$, we can differentiate $v(t)$ and use both the transport equation \fref{diffusioneq} and item $(iii)$ of Lemma \ref{lemL} to arrive at
\begin{align*}
& \|\psi^\eps(t)-\big<\psi^\eps(t),\mathbf{1}\big>_{L^2( \tilde{ \sigma})} \|^2_{L^2(\Rm^{d+1}\times\Sm^d)}\leq v(t)=v(0)+\frac{2}{\eps^2}\int_0^t\big<\psi^\eps(s),\bar{\calL} \psi^\eps (s)\big>_{L^2(\Rm^{d+1}\times\Sm^d)}ds \\
& \hspace{2cm} \leq v(0)-\frac{C}{\eps^2} \int_0^t \|\psi^\eps(s)-\big<\psi^\eps(s),\mathbf{1}\big>_{L^2( \tilde{ \sigma})} \|^2_{L^2(\Rm^{d+1}\times\Sm^d)} ds,
\end{align*}
which yields,
\be
\label{convdiffest}
 \|\psi^\eps(t)-\big<\psi^\eps(t),\mathbf{1}\big>_{L^2( \tilde{ \sigma})}\|_{L^2(\Rm^{d+1}\times\Sm^d)}\leq e^{-C t/\eps^2} \|\phi\|_{L^2(\Rm^{d+1}\times\Sm^d)}.
\ee
This shows that $\psi^\eps$ converges strongly to its average in $k$. We show now that this average converges weakly to zero, which implies that $\psi^\eps$ converges weakly to zero as well. Let indeed $u^\eps=\psi^\eps-\big<\psi^\eps,\mathbf{1}\big>_{L^2( \tilde{ \sigma})}$, and define
$$
\Psi^\eps=\big<\psi^\eps ,\mathbf{1}\big>_{L^2( \tilde{ \sigma})}, \qquad \Psi^\eps(t=0)=0,
$$ 
which satisfies
$$
\eps \partial_t \Psi^\eps=- \int_{\Sm^d} k \cdot \nabla_x \psi^\eps  \; d \tilde{\sigma}(k)=- \int_{\Sm^d} k \cdot \nabla_x u^\eps \; d \tilde{\sigma}(k)
$$
since $\int_{\Sm^d} k \tilde{\sigma}(k)=0$. It then follows from \fref{convdiffest}, for all $t>0$ and for all $ \varphi \in \calC^1_c(\Rm^{d+1})$,
\be
\label{convdiffest2}
\lim_{\eps \to 0} \left| \int_{\Rm^{d+1}} \Psi^\eps(t,x) \varphi(x) dx \right| \leq C \lim_{\eps \to 0} \eps^{-1}\| \nabla_x \varphi \|_{L^2(\Rm^{d+1})} \int_0^t \| u^\eps(s)\|_{L^2(\Rm^{d+1}\times\Sm^d)} ds=0.
\ee
The proof is concluded by using the density of $\calC^1_c(\Rm^{d+1})$ in $L^2(\Rm^{d+1})$, and the fact that the $L^2$ norm in $x$ and $k$ of $\psi^\eps$ is uniformly bounded in $t$ and $\eps$.
\end{proof}

It thus only remains now to address the convergence of $\tilde f^\eps$. When $f^0 \in \calC^0_b(\Rm^{d+1} \times \Sm^{d+1})\cap L^2(\Rm^{d+1} \times \Sm^{d+1})$, we conclude from dominated convergence and the convergence in law of $Y_{x}^\eps$ to $Y_{x}$ that for all $(t,x,k)\in(0,+\infty)\times \Rm^{d+1} \times \Sm^d$, $\tilde{f}_\eps(t,x,k)$ converges pointwise to 
$$\tilde f(t,x,k)=\int_{\Sm^d} d\tilde{\sigma}(p)  \E^{Y_x}[f^0(y_t,p)],$$
where we recall that $\E^{Y_x}$ denotes expectation with respect to the law of $Y_x$. Note that this limiting law is independent of $k$ since the diffusion coefficient $D$ is itself independent of $k$, and we can write
$$
\tilde f(t,x,k)=\tilde f(t,x)= \E^{Y_x}[\tilde f^0(y_t)], \qquad \textrm{where} \quad \tilde f^0(x)=\int_{\Sm^d} d\tilde{\sigma}(p)f^0(x,p).
$$
The latter form is a probabilistic representation of the unique solution to the diffusion equation $\partial_t u=\tilde{\calL_0} u$ with initial condition $u(0,x)=\tilde f^0(x)$, and therefore $\tilde f(t,x)$ is the unique solution to \fref{eqdiffusion}. When $f^0$ is only an $L^2$ function, we consider a regularized version of it denoted by $\tilde f^0_n$. As above, the associated $\tilde f^\eps_n$ converges pointwise in $(t,x)$ to $ \E^{Y_x}[\tilde f^0_n(y_t)]$. It is then not difficult to show that solutions to \fref{eqdiffusion} satisfy an estimate of the form $\| f(t)\|_{L^2} \leq \| f(0)\|_{L^2}$, and this, together with the latter pointwise convergence and the fact that the semigroup $T^\eps_t$ is continuous in $L^2$ show that, for all $t>0$,
$$
\int_{\Rm^{d+1} \times \Sm^d} dxd\sigma f^\eps(t,x,k)\varphi(x,k) \to \int_{\Rm^{d+1} \times \Sm^d} dxd\sigma \tilde f(t,x)\varphi(x,k), \qquad \forall \varphi \in \calC^0_c(\Rm^{d+1} \times \Sm^d).
$$
Since $\calC^0_c(\Rm^{d+1} \times \Sm^d)$ is dense in $L^2(\Rm^{d+1} \times \Sm^d)$, this implies that $f^\eps(t,x,k)$ converges, for all $t>0$, weakly in $L^2$ to $f(t,x)$. This concludes the proof of Theorem \ref{th2}.

\section{Proof of Theorem \ref{th3}}\label{proofth3}

The convergence of $(f_{\eps})_\eps$, as well as the one of the measure $\mathbb{X}^\eps_{x,k}$, is obtained via the convergence of the law of $m^\eps$ starting at $k$, that we denote by $\mathbb{M}^\eps_{k}$. In fact, we have
\[
f_\eps(t,x,k)= \E_k\Big[f^0\Big(x-\int_0^t m^\eps(u) du, m^\eps(t)\Big)\Big]= \E^{\mathbb{M}^\eps_{k}}[f^0(\Theta_x(t)],\]
where 
\[
\begin{array}{rcl}
\Theta_x:\,\calD([0,+\infty),\Sm^d)&\longrightarrow &\calD([0,+\infty),\Rm^{d+1} \times \Sm^d)\\
\omega &\mapsto & t\mapsto\big(  x-\int_0^t \omega(u) du,\omega(t)\big)
\end{array}
\]
is a continuous application for the Skorohod topology. Note that we have $\mathbb{X}^\eps_{x,k}=\mathbb{M}^\eps_{k}\circ \Theta^{-1}_x$.
\begin{lemma}\label{relatcomp}
$(\mathbb{M}^\eps_{k})_\eps$ is tight in $\mathcal{D}([0,+\infty), \Sm^d)$. 
\end{lemma}

\begin{proof}
We will prove this lemma by using Aldous's tightness criterion \cite[Theorem 16.10 pp. 178]{billingsley1}. The first requirement is 
\[\lim_{M\to +\infty}\sup_{\eps \in (0,1)}\mathbb{M}^\eps_{k}\Big(\sup_{0\leq t\leq T} \|y_t \|_{\Rm^{d+1}} \geq M\Big)=0\qquad \forall T>0,\]
which is direct since $\mathbb{M}^\eps_{k}$ is supported on $\Sm^d$. Now, let $T>0$ and $\tau$ be a discrete stopping time relatively to the canonical filtration, and bounded by $T$. Considering also two real numbers $\mu >0$ and $\nu >0$, we have
\[\mathbb{M}^\eps_{k}(\| y_{\tau +\mu}-y_\tau\|_{\Rm^{d+1}} >\nu)\leq \sum_{j=1}^{d+1} \mathbb{M}^\eps_{k}( \vert y^j_{\tau +\mu}-y^j_{\tau} \vert >\nu/(d+1) ), \]
where $y^j$ stands for the $j$-th component of $y$. In order to study the increment, let us us introduce 
\bee M_j(t)&=&y^j_t-\left(\int_0^t du \bar \calL_\eps (p-y_u) \right)_j \\
&=&y^j_t-\int_0^t du \textrm{ p.v.} \int_{\Sm^d} d\sigma(p) F^\eps(p\cdot y_u)(p_j-y^j_u) 
\eee
which is a martingale under $\mathbb{M}^\eps_{k}$ (see \cite[Proposition 1.7 pp. 162]{ethier}) with quadratic variation
\be\label{defquad}<M_j>(t)=\int_0^t (\bar\calL_\eps \varphi^2 (y^j_u) -2\varphi(y^j_u) \bar\calL_\eps\varphi(y^j_u))du,\ee
where $\varphi(x)=x$. Hence,
\[\begin{split}
\mathbb{M}^\eps_{k}( \vert y^j_{\tau +\mu}&-y^j_\tau \vert >\nu/(d+1) )\\
& \leq \mathbb{M}^\eps_{k}( \vert M_j(\tau +\mu)-M_j(\tau) \vert >\nu/(2(d+1)))\\
&+\mathbb{M}^\eps_{k}\Big(\Big\vert\int^{\tau +\mu}_{\tau} du \textrm{ p.v.}\int_{\Sm^d} d\sigma(p) F^\eps(p\cdot y_u)(p_j-y^j_u)  \Big\vert>\nu/(2(d+1))\Big)\\
&\leq I+II.
\end{split}\]
Regarding the second term, proceeding as in the proof of Lemma \ref{coefD}, we find
\be\label{calc1}\begin{split}
\Big\vert \int^{\tau +\mu}_{\tau} du \textrm{ p.v.} \int_{\Sm^d} d\sigma(p) F^\eps(p\cdot y_u)(p_j-y^j_u))\Big\vert \leq \mathfrak{C}\int^{\tau +\mu}_{\tau} du |y^j_u| \leq  \mu \mathfrak{C},
\end{split}\ee
for all $\mu>0$, and therefore $II=0$. Now, regarding the term $I$, we have using the martingale property of $M_j$
\[\begin{split}
 \mathbb{M}^\eps_{k}( \vert M_j(\tau +\mu)&-M_j(\tau) \vert^2 >\nu^2/(4(d+1)^2))\\
 &\leq \frac{4(d+1)^2}{\nu^2} \E^{ \mathbb{M}^\eps_{k}}[ \vert M_j(\tau +\mu)-M_j(\tau) \vert^2]\\
 &\leq \frac{4(d+1)^2}{\nu^2} \E^{ \mathbb{M}^\eps_{k}}[ <M_j>(\tau +\mu)-<M_j>(\tau) ]\\
 &\leq   C_{2,\nu} \mu,
 \end{split}\]
where we used \eqref{defquad} and similar arguments as in the derivation of \eqref{calc1}. As a result, we obtain
\[ \lim_{\mu\to 0}\sup_{\eps}\sup_{\tau} \mathbb{M}^\eps_{k}(\| y_{\tau +\mu}-y_\tau\|_{\Rm^{d+1}} >\nu)=0,\]
which concludes the proof of Lemma \ref{relatcomp}.
\end{proof}

We will use the following lemma that shows the convergence of the generators.  
\begin{lemma}\label{cvgene}
We have for all $\varphi \in \mathcal{C}^\infty(\Sm^{d})$, 
\[\lim_{\eps \to 0}\|\mathcal{L}_{\eps}\varphi-\mathcal{L}_\beta \varphi\|_{L^\infty(\Sm^{d})}=0.\]
\end{lemma}

\begin{proof} Let $\delta F^\eps (s):= \eps^{\beta+d/2}K(\eps(1-s))-a_1\vert 1-s\vert ^{-\beta-d/2}$. According to \eqref{equivpeak}, for all $s\in(-1,1)$ and for all $\mu>0$, there exists $\eps_0>0$ such that, $\forall \eps \in (0,\eps_0)$, 
\be \label{estk} 
|\delta F^\eps (s)|\leq \mu \vert 1-s\vert ^{-\beta-d/2}.\ee
We then write
$$
 \delta \calL_\eps \varphi:=\calL_\eps \varphi-\mathcal{L}_{\beta}\varphi=\textrm{ p.v.} \int_{\Sm^d} d\sigma(p) \delta F^\eps(p\cdot k)(\varphi(p)-\varphi(k)). 
$$
We may assume without loss of generality that $k=e_{d+1}=(0,\dots,0,1)$, and write
\[
p = (\sqrt{1-s^2}u,s),
\]
with $s\in[-1,1]$, and $ u\in \Sm^{d-1}$. Then,
$$
\delta \calL_\eps \varphi(k)=\lim_{\eta \to 0} \int_{\Sm^{d-1}} \int_{-1}^{1-\eta} \delta F_\eps(s) (\varphi(\sqrt{1-s^2}u+s k)-\varphi(k)) (1-s^2)^\frac{d-2}{2}d\sigma(u) ds.
$$
 Recasting $\varphi$ as
$$
\varphi(k)=\phi(0,\cdots,0,1), \qquad \varphi(\sqrt{1-s^2}u+s k)=\phi(\sqrt{1-s^2}u_1, \cdots,\sqrt{1-s^2}u_{d},s),
$$
we have
\bee
\varphi(\sqrt{1-s^2}u+s k)-\varphi(k)&=&(s-1) \partial_{x_{d+1}} \phi(0,\cdots,0,1)+ \sqrt{1-s^2} u \cdot \nabla_{d}\phi(0,\cdots,0,1)\\
&&+\calO(|s-1|),
\eee
where $\partial_{x_{d+1}}$ denotes partial derivation with respect to the $(d+1)-$th variable and $\nabla_{d}$ the gradient with respect to the $d$ first variables. Since $\int_{\Sm^{d-1}} u d\sigma(u)=0$, it follows from the equation above and \fref{estk} that
\be
|\delta \calL_\eps f(k)| \leq  C\mu \int_{-1}^{1} \frac{(1-s)(1-s^2)^\frac{d-2}{2}}{(1-s)^{\beta+\frac{d}{2}}} ds \leq C\mu \int_{-1}^{1} \frac{ds}{(1-s)^\beta} \leq C \mu,
\ee
since $\beta \in (0,1)$. This concludes the proof of the lemma.
\end{proof}

We have all needed now to conclude the proof. We deduce from Lemma \ref{relatcomp} that, up to the extraction of a subsequence, $\mathbb{M}^\eps_{k}$ converges weakly to a measure $\mathbb{M}_{k}$.  As a consequence of Lemma \ref{cvgene}, $\mathbb{M}_{k}$ is a solution to the martingale problem associated to $\bar{\calL}_\beta$ and starting at $k$. Since the equation $\partial_t u= \bar \calL_\beta u$ admits a unique solution for a given initial condition, it turns out that this martingale problem is well-posed (c.f. \cite[Theorem 4.2 pp. 184]{ethier}), so that the entire sequence converges since the limiting measure is unique.

When $f_0 \in \calC_b^0(\Rm^{d+1} \times \Sm^d) \cap L^2(\Rm^{d+1} \times \Sm^d)$, then, pointwise in $(t,x,k)$,
$$
f_\eps(t,x,k)=  \E^{\mathbb{M}^\eps_k\circ\Theta^{-1}_x}[f^0(y_t)] \to  \E^{\mathbb{X}_{x,k}}[f^0(y_t)],
$$
where $\mathbb{X}_{x,k}=\mathbb{M}_k\circ\Theta^{-1}_x$ is the law of a diffusion process with generator $-k\cdot\nabla_x+\bar{\calL}_\beta$ and starting at $(x,k)$. This, together with Lemma \ref{cvgene} and dominated convergence, allows us to pass to the limit in the weak formulation of \fref{peak}, and to deduce that the limit above is a  weak solution to \fref{FP}. According to Theorem \ref{th1}, the solution to \fref{FP} is unique and actually $\calC^\infty$.

When $f_0$ is only in $L^2(\Rm^{d+1} \times \Sm^d)$, we use a regularization procedure similar to the one at the end of the proof of Theorem \ref{th2}. This concludes the proof of Theorem \ref{th3}.

 \end{document}